\documentclass{amsart}
\usepackage{xcolor,framed}
\colorlet{shadecolor}{blue!20}
\usepackage[all]{xy}

\newtheorem{theorem}{Theorem}
\newtheorem{lemma}{Lemma}
\newtheorem{remark}{Remark}
\newtheorem{example}{Example}
\theoremstyle{definition}
\newtheorem{definition}{Definition}

\newcommand{\Arc}{\operatorname{Arc}}
\newcommand{\Dens}{\operatorname{Dens}} 
\newcommand{\Diff}{\operatorname{Diff}}
\newcommand{\Emb}{\operatorname{Emb}}
\newcommand{\Hor}{\operatorname{Hor}}
\newcommand{\Imm}{\operatorname{Imm}}
\newcommand{\Nor}{\operatorname{Nor}}
\newcommand{\Tan}{\operatorname{Tan}}
\newcommand{\Ver}{\operatorname{Ver}}
\newcommand{\bR}{\mathbb R}
\newcommand{\fX}{\mathfrak X}
\newcommand{\rd}{\mathrm d}
\newcommand{\grad}{\operatorname{grad}}


\begin{document}

\title{H\"ormander's condition for normal bundles on spaces of immersions}
\author{Martin Bauer}
\address{Fakult\"at f\"ur Mathematik und Geoinformation, TU Wien}
\email{bauer.martin@univie.ac.at}
\author{Philipp Harms}
\address{Department of Mathematics, ETH Z\"urich}
\email{philipp.harms@math.ethz.ch}
\date{September 2015}
\thanks{All participants of the ``Math on the Rocks'' workshop in Grundsund, Sweden, contributed to the results in these notes. In particular, Lemma~\ref{lem:forms} in its current form is in large parts due to Olivier Verdier. M. Bauer was supported by the European Research Council (ERC), within the project 306445 (Isoperimetric Inequalities and Integral Geometry) and by the FWF-project P24625 (Geometry of Shape spaces).}
\begin{abstract}
Several representations of geometric shapes involve quotients of mapping spaces. The projection onto the quotient space defines two sub-bundles of the tangent bundle, called the horizontal and vertical bundle. We investigate in these notes the sub-Riemannian geometries of these bundles. In particular, we show for a selection of bundles which naturally occur in applications that they are either bracket generating or integrable.	
\end{abstract}
\maketitle

\section{Introduction}

Several representations of geometric shapes involve quotients of mapping spaces. Three examples are presented in the diagram below:
\begin{equation*}
\xymatrix{
\Imm(M,N) \ar[d] 	
& \Diff(N) \ar[d] 
& \Diff(M) \ar[d]
\\
\Imm(M,N)/\Diff(M)
& \Emb(M,N) 
& \Dens(M)=\Diff(M)/\Diff_\mu(M)
}
\end{equation*}
The first example is the quotient of embeddings modulo reparametrizations. In the second example, $\Diff(N)$ acts on some fixed ``template'' element of $\Emb(M,N)$. The third example is Moser's representation of densities as diffeomorphisms modulo volume preserving diffeomorphisms \cite{Mo1965,EbMa1970}.

Let us abstract from these examples and consider a submersion $\pi\colon P \to Q$ between possibly infinite-dimensional manifolds. If $P$ is endowed with a (weak) Riemannian metric $G$, two natural sub-bundles of $TP$ appear: the vertical bundle $\Ver$ is defined as the kernel of $T\pi$ and the horizontal bundle $\Hor$ as the set of tangent vectors in $TP$ which are $G$-orthogonal to $\Ver$. Note that $\Ver + \Hor$ might or might not span all of $TP$. However,  any closed complement of the vertical bundle is the horizontal bundle of some Riemannian metric, as was recently shown in \cite{bauer2015metrics} for the special case of planar curves. 

While the vertical bundle is always integrable (the integral manifolds are the fibers of the projection), it is in general not clear whether the horizontal bundles 	are integrable or, at the other extreme, bracket generating. This question is interesting for several reasons.
\begin{itemize}
\item Integrability of the horizontal bundle is necessary for lifting loops in $Q$ to horizontal loops in $P$. This is a natural task in, for example, the analysis of cardiac cycles, which can be represented as loops in shape space. 
\item If the horizontal bundle is integrable, then the horizontal geodesic equation can be solved in the lower-dimensional coordinate system of the integral manifold instead of the higher-dimensional coordinate system of $P$. 
\item If on the other hand the horizontal bundle is bracket generating, then its integral manifold is a dense subset of $P$, and any two points in the integral manifold can be connected by a horizontal curve \cite{Led2004,GMV2012}. 
\end{itemize}

\section{Results}

\begin{definition}
Let $M$ be a compact manifold and $(N,\overline g)$ a Riemannian manifold. Then the sub-bundles $\Tan$ and $\Nor$ of $T\Imm(M,N)$ are given at each $f\in\Imm(M,N)$ by 
\begin{align*}
\Tan_f &= \{Tf \circ X: X \in \fX(M)\}, 
\\
\Nor_f &= \{h \in T_f\Imm(M,N): \forall x \in M, \forall X \in T_xM, \overline g(h(x), Tf(X))=0\}.
\end{align*}
\end{definition}

\begin{remark}
The bundle $\Tan$ is the vertical bundle of the projection onto the space of unparametrized immersions and is integrable. Indeed, the group $\Diff(M)$ acts on $\Imm(M,N)$ by composition from the right and the $\Diff(M)$-orbits are integral manifolds for $\Tan$. 
\end{remark}

The following theorem shows that $\Nor$ is bracket generating and that the first bracket is enough to generate all of the tangent space.

\begin{theorem}\label{thm:nor}
Let $M$ be compact and $\dim(N)=\dim(M)+1$. Then 
$$\Nor + [\Nor,\Nor] = T\Imm(M,N).$$
\end{theorem}

\begin{proof}
Assume that a normal vector field $n$ to $f\in \Imm(M,N)$ is defined on all of $M$ and locally around $f$. Then any functions $a,b$ on $M$ define local vector fields $an,bn$ on $\Imm(M,N)$. Let $\nabla$ denote the covariant derivative on $\Imm(M,N)$ which is associated to $(N,\overline g)$; see Section~\ref{sec:covariant}. Then the Lie bracket $[an,bn]$ can be expressed using covariant derivatives because $\nabla$ is torsion-free by Lemma~\ref{lem:torsion}. By the variational formula for the normal vector in Lemma~\ref{lem:variation}, 
\begin{equation}\label{eq:nor}\begin{aligned}[]
[an,bn] 
&= \nabla_{an}(bn)-\nabla_{bn}(an) 
=b \nabla_{an}n-a \nabla_{bn}n
= Tf\left(a\grad^g b-b\grad^g a\right)
\\
&= Tf\left(g^{-1}(a\rd b-b\rd a)\right).
\end{aligned}\end{equation}
By Lemma~\ref{lem:forms}, all one-forms on $M$ are linear combinations of one-forms $a\rd b-b\rd a$. This shows $\Tan \supseteq [\Nor+\Nor]$.

The assumption that $n$ is defined globally on $M$ can be eliminated by localization. Indeed, as $M$ is compact, any vector field $X \in \mathfrak X(M)$ is a finite sum of vector fields supported in domains $U$ such that $n$ is defined in a neighborhood of $\overline U$. By Remark~\ref{rem:support}, the functions $a,b$ can be chosen with support in $\overline U$. It follows that $an,bn$ are well-defined.
\end{proof}

\begin{remark}
If $N$ is a linear space, the covariant derivatives in the proof of Theorem~\ref{thm:nor} can be replaced by directional derivatives. For example, in the case of planar curves $M=S^1,N=\bR^2$, Equation~\eqref{eq:nor} evaluated at some $c \in \Imm(S^1,\bR^2)$ takes the form
$$[an,bn]_c = D_{c,an}(bn)-D_{c,bn}(an)
=(aD_s b-bD_s a)v.$$
\end{remark}

\begin{example}
A further example of an integrable bundle are the arc-length preserving deformations of planar curves. Let $M=S^1, N=\bR^2$, and define the vector bundle $\Arc\subset T\Imm(S^1,\bR^2)$ at each $c\in\Imm(S^1,\bR^2)$ by 
$$\Arc_c = \{h \in T_c\Imm(S^1,\bR^2) \colon D_s (D_{c,h}ds/ds)=0\}.$$
Then the bundle $\Arc$ is integrable and the collections of curves $c$ whose velocities $|\partial_\theta c| \in C^\infty(S^1)$ are multiples of each other are integral manifolds for $\Arc$. 
\end{example}


\appendix

\section{Covariant derivative on $\Imm(M,N)$}\label{sec:covariant}

We recall some definitions and results of \cite[Section~4.2]{bauer2011sobolev}. Let $\nabla$ be the Levi-Civita covariant derivative of the Riemannian manifold $(N,\overline g)$. Then $\nabla_X h\colon Q \to TN$ is well-defined for any manifold $Q$, vector field $X \in \mathfrak X(Q)$, and mapping $h\colon Q \to TN$. This covariant derivative can be extended to $\Imm(M,N)$ using the isomorphism $\wedge\colon C^\infty(Q,C^\infty(M,TN)) \to C^\infty(Q\times M,TN)$ and its inverse $\vee$. Let $h\colon Q\to T\Imm(M,N)$ and $X \in \mathfrak X(Q)$. Then $\nabla_X h$ is defined as $(\nabla_{X\times 0}h^{\wedge})^\vee$. 

\begin{lemma}\label{lem:torsion}
The covariant derivative $\nabla$ on $\Imm(M,N)$ is torsion-free, i.e., $\nabla_XY-\nabla_YX=[X,Y]$ holds for any $X,Y \in \mathfrak X(\Imm(M,N))$.
\end{lemma}

\begin{proof}
Let $X,Y$ be vector fields on a manifold $Q$ and $f\colon Q\times M \to TN$. Then $X\times 0$ and $Y\times 0$ are vector fields on $Q\times M$ and 
$$\nabla_{X\times 0}Tf(Y\times 0)-\nabla_{Y\times 0}Tf(X\times 0)=Tf([X\times 0,Y\times 0])=Tf([X,Y]\times 0)$$
because the Levi-Civita covariant derivative on $(N,\overline g)$ is torsion-free \cite[Section~22.10]{michor2008topics}. The statement of the Lemma follows by setting $Q=\Imm(M,N)$, $f(g,x)=g(x)$ for all $g \in \Imm(M,N)$ and $x \in M$, and noting that $Tf(X\times 0)=X^\wedge$.
\end{proof}

\section{Variational formula for the normal vector field}

\begin{lemma}\label{lem:variation}
\cite[Section~4.11]{bauer2012almost}
Let $X$ be a vector field on $\Imm(M,N)$. Then the variation of the normal vector field $n$ in the direction of $X$ is
$$\nabla_X n = -Tf\left( L X^\top) 
+ \operatorname{grad}^g \overline g(X,n)\right),$$
where $X=Tf \circ X^\top + \overline g(X,n) n$ is the decomposition in tangential and normal components, $g=f^*\overline g$ is the pull-back of $\overline g$ to $M$, and $L$ is the Weingarten map. 
\end{lemma}

%

\section{Auxiliary result about one-forms}

\begin{lemma}\label{lem:forms}
Let $M$ be compact. Then $\Omega^1(M)=\operatorname{span}_{\bR} \{a \rd b-b \rd a: a,b \in C^\infty(M)\}.$
\end{lemma}

\begin{proof}
All one-forms $f\rd g$ with positive $f$ can be generated by elements of the form $a \rd b-b \rd a$. Indeed, set $a=(fe^{-g})^{1/2},  b=(fe^g)^{1/2}$ and check that $a\rd b-b\rd a=f\rd g$. Moreover, the closed one-form $\rd g$ can be generated by setting $a=1, b=g$. This allows one to generate all one-forms $f\rd g$ with $f$ bounded from below. As every function can be decomposed in a function bounded from below, and one from above, this allows one to generate all one-forms $f \rd g$. 

Let $U$ be an open set in $M$ such that there exist functions $x^1,\dots,x^d$ defined on all of $M$ providing a coordinate system on $U$. Then any one-from $\alpha$ with support in $U$ can be written as $\alpha = \sum_{i=1}^d \alpha(\partial x^i) \rd x^i$, showing that $\alpha$ is a linear combination of expressions of the form $f \rd g$.

Finally, any one-form on $M$ is a sum of finitely many one-forms supported in open sets $U$ as above. To see this, note that any point $x$ in $M$ has an open neighborhood $U$ with the above properties. As $M$ is compact, finitely many such neighborhoods $U_1,\dots,U_n$ cover $M$. Let $\phi_1,\dots,\phi_n$ be a partition of unity subordinate to $U_1,\dots,U_n$. Then any one-form $\alpha$ can be written as $\alpha = \sum_{i=1}^n \phi_i \alpha$ and $\phi_i\alpha$ is supported in $U_i$.
\end{proof}

\begin{remark}\label{rem:support}
If $\alpha$ is a one-form with support in an open set $U\subseteq M$, then it can be represented as a linear combination of forms $a\rd b-b\rd a$ with functions $a,b$ supported in $\overline U$. Indeed, at each step of the proof of Lemma~\ref{lem:forms}, the functions $a,b$ may be multiplied by a bump function which vanishes outside of $U$ and equals $1$ on the support of $\alpha$.
\end{remark}


\end{document}